\documentclass[12pt,reqno]{amsart}

 \RequirePackage{geometry,color}
\usepackage{enumerate}

\usepackage{graphicx}        
\usepackage{multicol}        
\usepackage[bottom]{footmisc}

\usepackage{amssymb}  

\usepackage{verbatim} 
\usepackage{amsxtra}    
\usepackage{graphics}
\usepackage{epstopdf}
\usepackage{caption}
\usepackage{subcaption}

\usepackage[mathscr]{eucal}
\usepackage{bm}

\allowdisplaybreaks

\newcommand{\set}[1]{\left\{#1\right\}}
\newcommand{\abs}[1]{\left\vert #1\right\vert}%
\newcommand{\norm}[1]{\left\Vert #1\right\Vert}%
\newcommand{\trinorm}[1]{\left\|\kern-1pt\left|#1\right|\kern-1pt\right\|}%
\newcommand{\To}{\longrightarrow}

\DeclareMathOperator{\supp}{\mathrm{supp}}

\newcommand{\De}{\Delta}

\newcommand{\ve}{\varepsilon}

\newcommand{\si}{\sigma}

\newcommand{\f}{\frac}

\newcommand{\nf}{\infty}
\newcommand{\q}{\quad}

\newcommand{\li}{L^\infty}
\newcommand{\wh}{\widehat}
\newcommand{\R}{\mathbb R}

\theoremstyle{definition}
\newtheorem{defn}{Definition}[section]

\theoremstyle{theorem}
\newtheorem{thm}[defn]{Theorem}

\newtheorem{lem}[defn]{Lemma}
\theoremstyle{remark}

\newcommand{\beq}{\begin{equation}}
\newcommand{\eeq}{\end{equation}}

\begin{document}

\title[The H\"ormander Multiplier Theorem, III]{The H\"ormander Multiplier Theorem, III: The complete bilinear case via interpolation}

\author[Grafakos]{Loukas Grafakos}
\address{Department of Mathematics,
University of Missouri,
Columbia, MO 65211, USA}
\email{grafakosl@missouri.edu}

\author[Nguyen]{Hanh Van Nguyen}
\address{Department of Mathematics,
The University of Alabama, 
Tuscaloosa, AL 35487, USA}
\email{hvnguyen@ua.edu}
\thanks{The first author would like to thank the Simons Foundation.}

\begin{abstract}
We develop a special multilinear complex interpolation theorem that allows us to 
prove an optimal version of the bilinear H\"ormander multiplier theorem concerning 
symbols that lie in the Sobolev space $L^r_s(\mathbb R^{2n})$, $2\le r<\infty$, $rs>2n$, 
uniformly over all annuli. More precisely, given a smoothness index $s$,
we find the largest open set of indices $(1/p_1,1/p_2 )$ for which we have boundedness for the associated
bilinear multiplier operator from $L^{p_1}(\mathbb R^{ n})\times L^{p_2} (\mathbb R^{ n})$ to $ L^p(\mathbb R^{ n})$
when $1/p=1/p_1+1/p_2$, $1<p_1,p_2<\infty$.
\end{abstract}

\maketitle


\section{Introduction}
Multipliers are linear operators of the form
$$
T_\si(f)(x) = \int_{\mathbb R^n} \wh{f}(\xi) \si(\xi) e^{2\pi i x\cdot \xi}d\xi\, ,
$$
where $f$ is a Schwartz function on $\mathbb R^n$ and
$\wh{f}(\xi) = \int_{\mathbb R^n} f(x)   e^{-2\pi i x\cdot \xi}dx$ is its Fourier transform.

Let $\Psi$ be a Schwartz function whose Fourier transform is supported in the annulus of the form
$\{\xi: 1/2< |\xi|< 2\}$ which satisfies $\sum_{j\in \mathbb Z} \wh{\Psi}(2^{-j}\xi)=1$ for all $\xi\neq 0$.
We denote by  $\De$   the Laplacian and by $(I-\De)^{s/2} $   the operator given on the Fourier transform
by multiplication by
$(1+4\pi^2 |\xi|^2)^{s/2}$; also  for $s>0$, and
we denote by $L^r_s$ the Sobolev  space of all functions $h$ on $\mathbb R^n$
with norm
$
\|h\|_{L^r_\gamma}:=\|(I-\Delta)^{s/2} h \|_{L^r} <\nf.
$
Extending an earlier result of Mikhlin~\cite{Mihl}, 
the optimal version of the H\"ormander  multiplier theorem  says that if 
\beq\label{1H}
\sup_{k\in \mathbb Z} \big\|\wh{\Psi}\si (2^k \cdot)\big\|_{L^{r}_s} <\nf
 \eeq
and
\beq\label{1c}
\Big| \f 1p -\f 12 \Big| <\f sn\, ,
 \eeq
 then $T_\si$ is bounded from $L^p(\mathbb R^n)$ to itself for $1<p<\infty$. 
 H\"ormander's~\cite{Hoe}
  original version of this theorem  stated  boundedness in the entire
 interval $1<p<\nf$ provided $s> n/2$. A   restriction on the indices 
first appeared in Calder\'on and Torchinsky~\cite{CT77}, while condition \eqref{1c}
    appeared in  \cite{GHH2016I}; this condition is sharp as 
examples are given in \cite{GHH2016I} indicating that the theorem   fails in general when 
 $\big| \f 1p -\f 12 \big| > \f sn$. Moreover, recently Slav\'\i kov\'a~\cite{slav} provided an 
 example showing that boundedness may also fail even on the critical 
 line $\big| \f 1p -\f 12 \big| = \f sn$.

 In this paper we provide bilinear analogues of these results.    The study of the 
 H\"ormander multiplier theorem in the multilinear setting was 
 initiated by Tomita~\cite{Tomit10} and was further studied by  Fujita,
 Grafakos, Miyachi,  Nguyen, Si, Tomita (see
\cite{FT12},   ~\cite{GMNT16}  ~\cite{GS12},  \cite{GMT13}, ~\cite{MT13}, ~\cite{MT14}) 
among others. 
 For a given function $\si$ on $\mathbb R^{2n}$ we define a bilinear operator
 $$
 T_\si(f_1,f_2)(x) = \int_{\mathbb R^n} \int_{\mathbb R^n}\wh{f_1}(\xi_1)\wh{f_2}(\xi_2) \si(\xi_1,\xi_2)
  e^{2\pi i  x\cdot ( \xi_1+\xi_2)}d\xi_1d\xi_2
 $$
originally defined on pairs of $\mathscr C_0^\infty$ functions $f_1,f_2$ on $\mathbb R^n$.
We fix a Schwartz function $\Psi$ on $\mathbb R^{2n}$ whose Fourier transform is supported in the annulus
$1/2\le |( \xi_1,\xi_2) |\le 2$ and satisfies
$$
\sum_{j\in \mathbb Z} \wh{\Psi}(2^{-j} ( \xi_1,\xi_2 ))= 1,\q\q    ( \xi_1,\xi_2) \neq 0.
$$

The following theorem is the main result of this paper:

\begin{thm}\label{main}
Let $2\le r<\infty$,  $s>\f{2n}{r}$,
  $1<p_1,p_2\le \infty$ and let $1/p=1/p_1+1/p_2>0$.   \\
(a) Let $n/2<s\le n$. Suppose  that
\begin{equation}\label{0098N}
\frac{1}{p_1}<\frac{s}{n},\, 
\frac{1}{p_2}<\frac{s}{n},\, 
1-\frac{s}{n}<\frac{1}{p}<\frac{s}{n}+\frac{1}{2} .
\end{equation}
Then for all  $\mathscr C_0^\infty(\mathbb R^n)$  functions $f_1,f_2$ we have
\begin{equation}\label{0098}
  \|T_{\sigma}(f_1,f_2)\|_{L^p(\mathbb{R}^n)}\le C\sup_{j\in\mathbb{Z}}
  \|\sigma(2^j\cdot)\widehat{\Psi}\|_{L^r_s(\mathbb{R}^{2n})}\|f_1\|_{L^{p_1}(\mathbb{R}^n)}\|f_2\|_{L^{p_2}(\mathbb{R}^n)}.
\end{equation}
Moreover, if \eqref{0098} holds for all $f_1,f_2 \in \mathscr C_0^\infty$ and all $\si$ 
satisfying \eqref{1H},
then we must necessarily have 
\begin{equation}\label{0098LE}
\frac{1}{p_1}\le \frac{s}{n},\, 
\frac{1}{p_2}\le \frac{s}{n},\, 
1-\frac{s}{n}\le \frac{1}{p}\le \frac{s}{n}+\frac{1}{2} .
\end{equation}
(b) Let $n <s\le 3n/2$ and satisfy
\begin{equation}\label{0099N}
 \f{1}{p }  <\f sn+\f12\, .
\end{equation}
Then   \eqref{0098}  holds.
Moreover, if \eqref{0098} holds for all $f_1,f_2 \in \mathscr C_0^\infty$ and all $\si$
satisfying \eqref{1H},
 then we must necessarily have
\begin{equation}\label{0099LE}
 \f{1}{p } \le \f sn+\f12\, .
\end{equation}
(c) If $s>\frac{3n}{2}$ then \eqref{0098} holds for all $1<p_1,p_2<\infty$ and $\frac12 < p<\infty$.
\end{thm}

This theorem uses two main tools: First, 
 the optimal $n/2$-derivative result  in the local $L^2$-case contained
 in  \cite{GHH2016II}
  and a special type of multilinear interpolation suitable for the purposes of this problem (see Theorem \ref{MulInterp} below).
Figure~\ref{Fig01} (Section 4), plotted on a slanted $(1/p_1,1/p_2)$ plane,  shows the regions of boundedness for $T_\si$
in the two cases $n/2<s\le n$ and $n <s\le 3n/2$.
Note also that in the former case, the condition $1-\frac{s}{n}<\frac{1}{p}$ is only needed when $p>2$.

Finally, we mention that the necessity  of   conditions 
\eqref{0098N}, \eqref{0098LE}, and \eqref{0099LE} in Theorem~\ref{main} 
  are   consequences of Theorems 2 and 3 in \cite{GHH2016II}; these say that if boundedness 
  holds, then we must necessarily have 
\[
\frac{1}{p_1}\le \frac{s}{n},\quad
 \frac{1}{p_2}\le \frac{s}{n},\quad
 \frac{1}{p}\le \frac{s}{n} +\frac{1}{2}.
\]
Also, if $T_{\sigma}$ maps $L^{p_1}\times L^{p_2}$ to $ L^p$ and $p>2$,
then   duality implies that  $T_{\sigma}$ maps $L^{p'}\times L^{p_2}$ to $L^{p_1'}$.
Now $p'$ plays the role of $p_1$ and so constraint $\f 1 {p_1} \le \f s n$ becomes $1-\frac{s}{n}\le \frac{1}{p}$.
This proves \eqref{0098LE}. So the main contribution of this work is the sufficiency of the conditions in
\eqref{0098N} and \eqref{0099N}.

\section{Preliminary material for interpolation}

In this section we briefly discuss three lemmas needed in our interpolation.
\begin{lem}\label{lem:016}
  Let $0< p_0<p<p_1<\infty$ be related as in $1/p=(1-\theta)/p_0+\theta/p_1$ for some $\theta \in (0,1)$.
  Given $f\in {\mathscr C}_0^\nf(\mathbb R^n)$ and   $\ve>0,$ there exist
  smooth functions $h_j^\ve$, $j=1,\dots, N_\ve$, supported in  cubes   on $\mathbb R^n$ with pairwise disjoint interiors,
  and nonzero complex constants $c_j^\ve$ such that  the functions
\begin{equation}\label{EQ-Fz}
  f^{z,\ve}  =  \sum_{j=1}^{N_\ve} |c_j^\ve|^{\frac p{p_0} (1-z) +   \frac p{p_1}  z}  \, h_j^\ve
\end{equation}
satisfy
$$
 \big\|{f^{\theta,\ve}-f}\big\|_{L^2}+ \big\|{f^{\theta,\ve}-f}\big\|_{L^{p_0}}^{\min(1,p_0)}+ \big\|{f^{\theta,\ve}-f}\big\|_{L^{p_1}} ^{\min(1,p_1)}<  \ve
 $$
and
$$
  \|f^{it,\ve}\|_{L^{p_0}}^{p_0} \le   \|f \|_{L^p}^p +\ve'  \, , \q
  \|f^{1+it,\ve}\|_{L^{p_1}}^{p_1} \le     \|f \|_{L^p}^p  +\ve'\, ,
$$
where $\ve'$ depends on $\ve,p_0,p_1,p, \|f\|_{L^p}$ and tends to zero as $\ve\to 0$.
\end{lem}

\begin{proof}
Given $f\in {\mathscr C}_0^\nf(\mathbb R^n)$  and   $ \ve>0$, by   uniform continuity
there are $N_\ve$   cubes $Q_j^\ve$ (with disjoint interiors) and nonnegative  constants $c_j^\ve $ such that
$$
\Big\| f - \sum_{j=1}^{N_\ve} c_j^\ve \chi_{Q_j^\ve} \Big\|_{L^2} +
\Big\| f - \sum_{j=1}^{N_\ve} c_j^\ve \chi_{Q_j^\ve} \Big\|_{L^{p_0}} ^{\min(1,p_0)}   
+\Big\| f - \sum_{j=1}^{N_\ve} c_j^\ve \chi_{Q_j^\ve} \Big\|_{L^{p_1}} ^{\min(1,p_1)}  <\ve    \, .  
$$

 Find nonnegative smooth functions   $g_j^\ve\le \chi_{Q_j^\ve}$ such that
$$
\Big\| \sum_{j=1}^{N_\ve} c_j^\ve (  g_j^\ve-    \chi_{Q_j^\ve}   )  \Big\|_{L^2} +
\Big\|  \sum_{j=1}^{N_\ve}  c_j^\ve ( g_j^\ve-    \chi_{Q_j^\ve}   ) \Big\|_{L^{p_0}} ^{\min(1,p_0)}+
\Big\|  \sum_{j=1}^{N_\ve}  c_j^\ve ( g_j^\ve-    \chi_{Q_j^\ve}   ) \Big\|_{L^{p_1}}^{\min(1,p_1)} <\ve   
$$
and
$$
\bigg( \sum_{j=1}^{N_\ve} |c_j^\ve|^p \| g_j^\ve - \chi_{Q_j^\ve}\|_{L^{p_0}}^{p_0} \bigg)^{\f 1{p_0}}+
\bigg( \sum_{j=1}^{N_\ve} |c_j^\ve|^p \| g_j^\ve - \chi_{Q_j^\ve}\|_{L^{p_1}}^{p_1} \bigg)^{\f 1{p_1}} <\ve\, . 
$$
Let $\phi_j^\ve$ be the argument of the complex number $c_j^\ve$. 
Set $h_j^\ve = e^{i\phi_j^\ve} g_j^\ve$ and
notice that 
$ f_\theta^\ve = \sum_{j=1}^{N_\ve} |c_j^\ve| h_j^\ve =  \sum_{j=1}^{N_\ve}  c_j^\ve  g_j^\ve $  satisfies 
$$
 \big\|{f_\theta^\ve-f}\big\|_{L^2}+ \big\|{f_\theta^\ve-f}\big\|_{L^{p_0}}^{\min(1,p_0)}+ \big\|{f_\theta^\ve-f}\big\|_{L^{p_1}} ^{\min(1,p_1)}<  \ve .
 $$
Moreover, the choice of $g_j^\ve$   implies that
$$
\| f_{it} \|_{L^{p_0}} \le \big( B^{\min(1,p_0)} +\ve ^{\min(1,p_0)} \big)^{\f{1}{\min(1,p_0)} }, 
$$
 where 
 $$
 B=   \Big\| \sum_{j=1}^{N_\ve} c_j^\ve \chi_{Q_j^\ve} \Big\|_{L^p}^{\f{p}{p_0}} \le 
\Big( \big(\ve^{\min(1,p )} +\|f\|_{L^p}^{\min(1,p )}  \big)^{ \f{1}{\min(1,p )} }  \Big)^{\f{p}{p_0}}.
 $$
 An analogous estimate holds for   $f_{1+it}$. 
 Given $a,c>0$ and $\ve>0$ set $\ve'=\ve'(\ve,a,c)= (\ve^a+c^a)^{1/a}-c$. Then $ (\ve^a+c^a)^{1/a} \le \ve'+c$ and $\ve'\to 0 $ as $\ve\to 0$. Then for a suitable $\ve'$ that only depends on 
 $\ve, p,p_0,p_1, \|f\|_{L^p}$, the preceding estimates give:
  $ \|{f_{it}^\ve}\|_{L^{p_0}}^{p_0} \le   \|f \|_{L^p}^p +\ve'$ and $ \|{f_{1+it}^\ve}\|_{L^{p_1}}^{p_1} \le   \|f \|_{L^p}^p +\ve'$, as claimed. 
\end{proof}

\begin{lem}\label{lem:analint}
For $z$ in the strip $a<\Re(z)<b$ and $x\in \mathbb R^n$,
let   $H(z,x)$ be analytic in $z$ and  smooth in $x$ and assume  that it satisfies
  $$
  \abs{H(z,x)}+\abs{\f{dH}{dz} (z,x)}\le H_*(x),\quad\forall a<\Re(z)<b,
  $$
 where $H_*$ is  a measurable function on $\mathbb R^n$.
Let $f$ be a  complex-valued smooth function on $\mathbb R^n$ such that
  $$
  \int_{\mathbb R^n}\max\set{\abs{f(x)}^a,\abs{f(x)}^b}\Big\{ 1+ \abs{\log(\abs{f(x)})} \Big\} {H_*(x)}\, dx <\infty .
  $$
  Then the function
  $$
  G(z) = \int_{\mathbb R^n}\abs{f(x)}^ze^{i \textup{Arg } f(x)}H(z,x)\,dx
  $$
  is analytic on the strip $a<\Re(z)<b$ and continuous up to its boundary.
\end{lem}

\begin{proof}
  Let $A = \set{x\ :\ f(x)\ne 0}.$ For $x\in A$ denote
  $$
  F(z,x) = \abs{f(x)}^ze^{i \textup{Arg } f(x)} H(z,x).
  $$
   Fix $a<\Re(z_0)<b $. Then $F(z,x)$ is analytic at $z_0$ for all $x\in A$ as 
  \begin{align*}
 & \lim_{z\to z_0}\dfrac{F(z,x)-F(z_0, x)}{z-z_0} \\
 & = \abs{f(x)}^{z_0}\log \abs{f(x)} \,\, e^{i \textup{Arg }  f(x)}
  H(z_0,x) + \abs{f(x)}^{z_0} e^{i \textup{Arg } f(x)}  \f{dH }{dz}(z_0,x)\, .
\end{align*}
We also have
  $$
  \abs{\dfrac{F(z,x)-F(z_0, x)}{z-z_0}}\le
  \max\set{\abs{f(x)}^a,\abs{f(x)}^b}\Big( 1+ \abs{\log \abs{f(x)} }  \Big)  H_*(x)
  $$
  for all $x\in A.$ By the Lebesgue dominated convergence theorem, the function $G$ is analytic and its derivative is
  $$
  G'(z)= \int_{\mathbb R^n} \bigg[ \abs{f(x)}^{z }\log(\abs{f(x)})e^{i \textup{Arg }  f(x)}
  H(z ,x) + \abs{f(x)}^{z } e^{i \textup{Arg } f(x)}  \f{dH }{dz}(z ,x)   \bigg]  dx
  $$
Also,   $G$ is   continuous up to the boundary $\Re(z)=a$ and $\Re(z) = b.$
\end{proof}

\begin{lem}[{\cite{Grafa14CFA}}]\label{lem:ThreeLines} 
  Let $F$ be analytic on the open strip $S=\set{z\in\mathbb{C}\ :\ 0<\Re(z)<1}$ and continuous on its closure.
  Assume that for all  $0\le \tau \le 1$ there exist functions $A_\tau$ on the { real} line such that
$$
    | F(\tau+it) |  \le A_\tau(t)    \qquad \textup{  for all $t\in\mathbb{R}$,}
$$
and suppose that there exist constants   $A>0$ and $0<a<\pi$ such that for all $t\in \mathbb R$ we have
  $$
0<  A_\tau(t) \le \exp \big\{ A e^{a |t|} \big\} \, .
  $$
 Then   for  $0<\theta<1 $   we have
  $$
  \abs{F(\theta )}\le \exp\left\{
  \dfrac{\sin(\pi \theta)}{2}\int_{-\infty}^{\infty}\left[
  \dfrac{\log |A_0(t ) | }{\cosh(\pi t)-\cos(\pi\theta)}
  +
  \dfrac{\log | A_1(t )| }{\cosh(\pi t)+\cos(\pi\theta)}
  \right]dt
  \right\}\, .
  $$
\end{lem}

In calculations it is crucial to note that
\begin{equation*}\label{ide}
\dfrac{\sin(\pi \theta)}{2}\int_{-\infty}^{\infty}
  \dfrac{dt }{\cosh(\pi t)-\cos(\pi\theta)}  =1-\theta\, , \q
  \dfrac{\sin(\pi \theta)}{2}\int_{-\infty}^{\infty}
  \dfrac{dt }{\cosh(\pi t)+\cos(\pi\theta)} = \theta.
\end{equation*}

\section{Multilinear interpolation}

In this section we prove the main tool needed to derive Theorem~\ref{main} by interpolation.   We denote by   $\vec \xi =(\xi_1,\dots, \xi_m)
$ elements of  $\mathbb R^{mn}$, 
where $\xi_j\in   \mathbb R^{n}$. 
 We fix a Schwartz function $\Psi$ on $\mathbb R^{mn}$ whose Fourier transform is supported in the annulus $1/2\le |\vec \xi \, |\le 2$ and satisfies
$$
\sum_{j} \wh{\Psi}(2^{-j}\vec \xi\, )= 1,\q\q   0\neq \vec \xi \in \mathbb R^{mn}.
$$

\begin{thm}\label{MulInterp}
Let $0<p^1_0,\dots , p^m_0 \le \nf$, $0<p^1_1,\dots , p^m_1 \le \nf$, $0<q_0,q_1\le \nf$, $0\le s_0,s_1<\nf$, $1<r_0,r_1<\infty $,
$0<\theta<1$, and let
$$
\f{1}{p^l}= \f{1-\theta}{p^l_0}+\f{\theta}{p^l_1}, \q \f{1}{q}= \f{1-\theta}{q_0}+\f{\theta}{q_1}, \q \f{1}{r}= \f{1-\theta}{r_0}+\f{\theta}{r_1}, \q
s= (1-\theta) s_0+ \theta s_1
$$
for $l=1,\dots , m$. Assume   $r_0s_0>mn$ and  $r_1s_1>mn$ and that
  \begin{equation*}
    \norm{T_\sigma(f_1,\ldots,f_m)}_{L^{q_k}(\R^n) }\le K_k\sup_{j\in\mathbb Z}\norm{\sigma(2^j\cdot)\widehat{\Psi}}_{L^{r_{k}}_{s_k}(\R^{mn})}
    \prod_{l=1}^m
    \norm{f_l}_{L^{p^{l}_k}(\R^n)}
  \end{equation*}
  for $k=0,1$ where $K_0,K_1$ are positive constants.
  Then we have the intermediate estimate:
  \begin{equation}\label{equ:TSgmMul}
    \norm{T_\sigma(f_1,\ldots,f_m)}_{L^{q}(\R^{ n})}\le C_* \, K_0^{1-\theta}K_1^{\theta}
    \sup_{j\in\mathbb Z}\norm{\sigma(2^j\cdot)\widehat{\Psi}}_{L^{r}_{s}(\R^{mn})}
    \prod_{l=1}^m\norm{f_l}_{L^{p^{l}}(\R^{ n})},
  \end{equation}
  where $C_*$ depends on all the indices, $\theta$, and the dimension.
\end{thm}

\begin{proof}

  Fix a smooth function $\widehat{\Phi}$ on $\mathbb R^{mn}$ such that $\supp(\Phi)\subset\big\{\frac14\le |{\vec \xi\, }|\le 4\big\}$ and $\wh{\Phi}\equiv 1$ on the support of the function $\widehat{\Psi}.$ Denote
  $
  \varphi_j = (I-\Delta)^{\frac{s}2}[\sigma(2^j\cdot)\widehat{\Psi}]
  $
  and define
\begin{equation}\label{100}
  \sigma_z(\vec \xi\,) = \sum_{j\in \mathbb Z}(I-\Delta)^{-\frac{s_0(1-z)+s_1 z}{2}}
  \left[
 |\varphi_j|^{r(\frac{1-z}{r_0}+\frac{z}{r_1})}e^{i \textup{Arg } (\varphi_j)}
  \right](2^{-j}\vec \xi\,)\widehat{\Phi}(2^{-j}\vec \xi\,).
\end{equation}
 This sum has only finitely many terms and we now estimate its $\li$ norm.

 Fix $\vec \xi \in \mathbb R^{mn}$. Then there is a $j_0$ such that $|\vec \xi \,| \approx 2^{j_0}$ and there are only two terms in the sum in
 \eqref{100}.  For these terms we estimate  the $\li$ norm of
 $(I-\Delta)^{-\frac{s_0(1-z)+s_1 z}{2}}
  \big[ |\varphi_j|^{r(\frac{1-z}{r_0}+\frac{z}{r_1})}e^{i \textup{Arg } (\varphi_j)} \big]$.
  For   $z=\tau+it$ with $0\le \tau\le 1$, let $s_\tau= (1-\tau)s_0+\tau s_1$ and $1/r_\tau = (1-\tau)/r_0+\tau /r_1$.
  By the Sobolev embedding theorem
  we have
\begin{align}
\  \Big\| (I-\Delta)^{-\frac{s_0(1-z)+s_1 z}{2}}&
  \big[ |\varphi_j|^{r(\frac{1-z}{r_0}+\frac{z}{r_1})}e^{i \textup{Arg } (\varphi_j)} \big]  \Big\|_{\li (\R^{mn})}\notag\\
\le  &\  C(r_\tau ,{s_\tau},mn)
\Big\| (I-\Delta)^{-\frac{s_0(1-z)+s_1 z}{2}}
  \big[ |\varphi_j|^{r(\frac{1-z}{r_0}+\frac{z}{r_1})}e^{i \textup{Arg } (\varphi_j)} \big]  \Big\|_{ L^{r_\tau}_{s_\tau} (\R^{mn}) }\notag \\
  \le  &\  C(r_\tau ,{s_\tau},n)\Big\| (I-\Delta)^{it \frac{s_0-s_1 }{2}}
  \big[ |\varphi_j|^{r(\frac{1-z}{r_0}+\frac{z}{r_1})}e^{i \textup{Arg } (\varphi_j)} \big]  \Big\|_{ L^{r_\tau}  (\R^{mn}) } \notag\\
  \le  &\  C'(r_\tau ,{s_\tau},mn)(1+|s_0-s_1 |\, |t|)^{mn/2+1} \Big\|
    |\varphi_j|^{r(\frac{1-z}{r_0}+\frac{z}{r_1})}e^{i \textup{Arg } (\varphi_j)}    \Big\|_{ L^{r_\tau}  (\R^{mn}) } \notag\\
      \le  &\  C''(r_0, r_1,s_0,s_1 , \tau ,mn)(1+ \, |t|)^{mn/2+1} \Big\|
    |\varphi_j|^{r(\frac{1-\tau}{r_0}+\frac{\tau}{r_1})}    \Big\|_{ L^{r_\tau} (\R^{mn})  } \notag\\
          =  &\  C''(r_0, r_1,s_0,s_1 , \tau ,mn)(1+ \, |t|)^{mn/2+1} \big\|    \varphi_j     \big\|_{ L^{r }  (\R^{mn}) } ^{r/r_\tau}\, .\notag
\end{align}
It follows from this that
 \begin{equation}\label{200}
\| \si_{\tau+it} \|_{\li(\R^{mn})} \le  C''(r_0, r_1,s_0,s_1 , \tau ,mn)(1+ \, |t|)^{mn/2+1}
         \Big(  \sup_{j\in\mathbb Z}\norm{\sigma(2^j\cdot)\widehat{\Psi}}_{L^{r}_{s}(\R^{mn})} \Big)^{r/r_\tau} \, .
   \end{equation}

  Let $T_{\sigma_z}$ be the family of operators associated to the multipliers $\sigma_z.$ Let $\ve $ be given.

\medskip

\textbf{Case I:} $\bm{\min(q_0,q_1)>1}$.
  Fix $f_l, g\in  \mathcal C_0^\nf(\mathbb{R}^n)$ and $0<p_0^l,p^l, p_1^l<\nf$,  $1<q_0',q',q_1'<\nf$.
  Given $\ve>0$, by Lemma \ref{lem:016} there exist functions $f_l^{z,\ve} $ and
  $ {g}^{z,\ve}$ of the form \eqref{EQ-Fz}
  such that
\begin{equation}\label{normftheta}
  \|{f_l^{\theta,\ve}-f_l}\|_{L^{p^l_1}}+  \|{f_l^{\theta,\ve}-f_l}\|_{L^{p^l_0} } <\ve,\quad \|{g^{\theta,\ve}-g}\|_{L^{q'}}<\ve,
\end{equation}
  and that for all $l=1,\dots , m$ we have
  \begin{align*}
 & \|{f_l^{it,\ve}}\|_{L^{p_0^l}}\le   \big( \norm{f_l}_{L^{p^l}}^{p^l}+\ve'\big)^{\frac {1}{p_0^l}} ,\quad
  \|{f_l^{1+it\ve}}\|_{L^{p_1^l}}\le   \big( \norm{f_l}_{L^{p^l}}^{p^l}+\ve'\big)^{\frac {1}{p_1^l}},\\
 & \|{g^{it,\ve}}\|_{L^{q_0'}}\le   \big( \norm{g}_{L^{q'}}^{q'}+\ve'\big)^{\frac {1}{q_0'}},\quad
  \norm{g^{1+it,\ve}}_{L^{q_1'}}\le  \big( \norm{g}_{L^{q'}}^{q'}+\ve'\big)^{\frac {1}{q_1'}}.
  \end{align*}
  Define
  $$
  \begin{aligned}
  F(z) =& \int_{\mathbb R^n} T_{\sigma_z}(f_1^{z,\ve},\ldots,f_m^{z,\ve}) {g}^{z,\ve}\; dx\\
  =&\int_{\mathbb R^{mn}} \sigma_z(\vec\xi\,)\widehat{f_1^{z,\ve}}(\xi_1)\cdots\widehat{f_m^{z,\ve}}(\xi_m)
  \widehat{g^{z,\ve}}(-(\xi_1+\cdots+\xi_m))\; d\vec\xi\\
  =& \sum_{j\in \mathbb Z}\int_{\mathbb R^{mn}} (I-\Delta)^{-\frac{s_0(1-z)+s_1 z}{2}}
  \left[
  |\varphi_j|^{r(\frac{1-z}{r_0}+\frac{z}{r_1})}e^{i \textup{Arg } (\varphi_j)}
  \right](2^{-j}\xi)\widehat{\Phi}(2^{-j}\vec\xi\,)\\
  &\times\Big(\prod_{l=1}^m\widehat{f_l^{z,\ve}}(\xi_l)\Big)\widehat{ {g}^{z,\ve}}(-(\xi_1+\cdots+\xi_m))\; d\vec\xi\\
  =& \sum_{j\in \mathbb Z}\int_{\mathbb R^{mn}}
  \bigg[|{\varphi_j}|^{r(\frac{1-z}{r_0}+\frac{z}{r_1})}e^{i \textup{Arg } (\varphi_j)}
  \bigg](2^{-j}\vec\xi\,)\\
 &\times (I-\Delta)^{-\frac{s_0(1-z)+s_1 z}{2}}\bigg[\widehat{\Phi}(2^{-j}\vec\xi\,)\Big(\prod_{l=1}^m\widehat{f_l^{z,\ve}}(\xi_l)\Big)\widehat{ {g}^{z,\ve}}(-(\xi_1+\cdots+\xi_m))\bigg](\vec\xi\,)\; d\vec\xi.
  \end{aligned}
  $$
  Notice that
  $$
(I-\Delta)^{-\frac{s_0(1-z)+s_1 z}{2}}\bigg[\widehat{\Phi}(2^{-j}\vec\xi\,)\Big(\prod_{l=1}^m\widehat{f_l^{z,\ve}}(\xi_l)\Big)\widehat{ {g}^{z,\ve}}(-(\xi_1+\cdots+\xi_m))\bigg](\vec\xi\,)
  $$
  is equal to a finite sum (over $k_1,\dots , k_m,l$) of terms the form
  $$
 |c_{k_1}^\ve|^{\frac {p_1}{p_1^0}(1-z) + \frac {p_1}{p_1^1} z}
 \cdots |c_{k_m}^\ve|^{\frac {p_m}{p_m^0}(1-z) + \frac {p_m}{p_m^1} z}   |d_l^\ve|^{\frac{ q'}{q_0'}(1-z) + \frac {q'}{q_1'} z}
  (I-\Delta)^{-\frac{s_0(1-z)+s_1 z}{2}}
  \left[\widehat{\Phi}(2^{-j}\cdot)  \zeta_{k_1,\dots , k_m,l} \right]  (\vec\xi\,) ,
  $$
  which we call $  H(\vec \xi,z)$,
  where $\zeta_{k_1,\dots , k_m,l}$ are Schwartz functions. Thus $  H(\vec \xi,z)$ is an analytic function in $z$.

  Lemma \ref{lem:analint} guarantees that $F(z)$ is analytic on the strip $0<\Re(z)<1$ and continuous up to the boundary. Furthermore, by H\"older's inequality,
  $$
  \abs{F(it)}\le \norm{T_{\sigma_{it}}(f_1^{it,\ve} , \dots ,f_m^{it,\ve}  )}_{L^{q_0}}\norm{g_{it}^\ve}_{L^{q_0'}},
  $$
and  noting that only the terms with $j=k-1,k,k+1$ survive in the sum in \eqref{100} for  $\sigma_{it}(2^k\cdot)\widehat{\Psi}$, 
the Kato-Ponce inequality 
\cite{GO,KP}
 applied 
as $\| (I-\De)^{s/2} (F\wh \Phi) \|_{L^{r_0}} \le C \| (I-\De)^{s/2} (F ) \|_{L^{r_0}}$ yields  
  \begin{align*}
  \quad  \|T_{\sigma_{it}}&(f_1^{it,\ve} , \dots ,f_m^{it,\ve}  )\|_{L^{q_0}} \\
  \le&   K_0\sup_{k\in\mathbb Z}
  \norm{\sigma_{it}(2^k\cdot)\widehat{\Psi}}_{L^{r_0}_{s_0}} \prod_{l=1}^m \|{f_l^{it,\ve}}\|_{L^{p_0^l}}\\
  \le &   C_{n,r_0,s_0} K_0\sup_{k\in\mathbb Z} 
  \big\| (I-\De)^{\f {s_0}{2} } (I-\Delta)^{-\frac{s_0(1-it)+s_1 it}{2}}
  \big[  |\varphi_k|^{r(\frac{1-it}{r_0}+\frac{it}{r_1})}e^{i \textup{Arg } (\varphi_k)}
  \big] \big\|_{L^{r_0} }\prod_{l=1}^m \|{f_l^{it,\ve}}\|_{L^{p_0^l}} \\
  \le&\, \, C(m,n,r_0,  s_0 ) (1+|s_1-s_0|\, |t|)^{\frac{mn}2+1}K_0\,
  \sup_{j\in\mathbb Z}\| {\varphi_j}\|_{L^r}^{\frac{r}{r_0}}
 \prod_{l=1}^m   \big( \norm{f_l}_{L^{p^l}}^{p^l}+\ve'\big)^{\frac {1}{p^l_0}} \\
  = &\, \,  C(m,n,r_0 ,   s_0, s_1 )  (1+  |t|)^{\frac{mn}2+1}         K_0\,
  \sup_{j\in\mathbb Z}
  \norm{(I-\Delta)^{\frac{s}2}[\sigma(2^j\cdot)\widehat{\Psi}]}_{L^r}^{\frac{r}{r_0}}
 \prod_{l=1}^m   \big( \norm{f_l}_{L^{p^l}}^{p^l}+\ve'\big)^{\frac {1}{p^l_0}} .
  \end{align*}
  Thus, for some constant $C=C(m,n,r_0,s_0,s_1)$ we have
  $$
  \abs{F(it)}\le C  (1+\abs{t})^{\frac{mn}2+1}K_0
  \sup_{j\in\mathbb Z}
  \norm{(I-\Delta)^{\frac{s}2}[\sigma(2^j\cdot)\widehat{\Psi}]}_{L^r}^{\frac{r}{r_0}}
  \big( \norm{g}_{L^{q'}}^{q'}+\ve'\big)^{\frac {1}{q_0'}}
   \prod_{l=1}^m   \big( \norm{f_l}_{L^{p^l}}^{p^l}+\ve'\big)^{\frac {p^l}{p^l_0}}.
  $$
  Similarly, we can choose the constant $C=C(m,n,r_1,s_0,s_1)$ above large enough so that
  $$
\abs{F(1+it)}\le C  (1+\abs{t})^{\frac{mn}2+1}K_1
  \sup_{j\in\mathbb Z}
  \norm{(I-\Delta)^{\frac{s}2}[\sigma(2^j\cdot)\widehat{\Psi}]}_{L^r}^{\frac{r}{r_1}}
  \big( \norm{g}_{L^{q'}}^{q'}+\ve'\big)^{\frac {1}{q_1'}}
   \prod_{l=1}^m   \big( \norm{f_l}_{L^{p^l}}^{p^l}+\ve'\big)^{\frac {1}{p^l_1}}.
  $$

Note that $F(z)$ is a combination of finite terms of the form
\[
\Lambda_{k_1,\dots , k_m,l}(z)
 \int_{\mathbb R^{mn}} \sigma_z(\vec\xi\,)\widehat{h_{j_1}^{1,\ve}}(\xi_1)\cdots\widehat{h_{j_m}^{m,\ve}}(\xi_m)
  \widehat{g_j^\ve}(-(\xi_1+\cdots+\xi_m))\; d\vec\xi,
\]
where $\Lambda_{k_1,\dots , k_m,l}(z) = |c_{k_1}^\ve|^{\frac {p_1}{p_1^0}(1-z) + \frac {p_1}{p_1^1} z}
 \cdots |c_{k_m}^\ve|^{\frac {p_m}{p_m^0}(1-z) + \frac {p_m}{p_m^1} z}   |d_l^\ve|^{\frac{ q'}{q_0'}(1-z) + \frac {q'}{q_1'} z}$ and $h_{j_1}^{1,\ve}$, $g_j^\ve$ are smooth functions with compact support. Thus for $z=\tau+it$,  $t\in\mathbb{R}$  and $0\le \tau \le  1$
it follows from \eqref{200} and from the definition of $F(z)$  that
$$
|F(z)|\le C(\tau,\epsilon,f_1,\ldots,f_m,g,r_l,p_l,q_0,q_1)
(1+ \, |t|)^{\f{mn}2+1}
         \Big(  \sup_{j\in\mathbb Z}\norm{\sigma(2^j\cdot)\widehat{\Psi}}_{L^{r}_{s}} \Big)^{r/r_\tau}
=A_\tau(t)
$$
As $A_\tau(t)\le
\exp(A e^{a|t|})  $,  the admissible growth hypothesis of Lemma~\ref{lem:ThreeLines} is satisfied. Applying Lemma~\ref{lem:ThreeLines} we obtain
\begin{equation}\label{Ftheta}
  \abs{F(\theta)}\le C\, K_0^{1-\theta} K_1^{\theta}\sup_{j\in\mathbb Z}
  \norm{(I-\Delta)^{\frac{s}2}[\sigma(2^j\cdot)\widehat{\psi}]}_{L^r}
 \big( \norm{g}_{L^{q'}}^{q'}+\ve'\big)^{\f 1{q'}}
   \prod_{l=1}^m   \big( \norm{f_l}_{L^{p^l}}^{p^l}+\ve'\big)^{\f{1}{p^l}}.
\end{equation}
  But
  $$
  F(\theta) = \int_{\mathbb R^{n}}
  T_{\sigma}(f_1^{\theta,\ve},\ldots,f_m^{\theta,\ve}) g^{\theta,\ve}\; dx
  $$
and then we have
\begin{align*}
 \int_{\mathbb R^n}
 T_{\sigma}(f_1,\ldots,f_m)g\; dx
 = F(\theta)
 &+\int_{\mathbb R^n} \big[ T_{\sigma}(f_1,\ldots,f_m)- T_{\sigma}(f_1^{\theta,\ve},\ldots,f_m^{\theta,\ve}) \big]g\; dx \\
 &+ \int_{\mathbb R^n} T_{\sigma}(f_1^{\theta,\ve},\ldots,f_m^{\theta,\ve})\big( g-g^{\theta,\ve} \big)\; dx .
\end{align*}
A telescoping identity yields 
\[
|T_{\sigma}(f_1,\ldots,f_m)- T_{\sigma}(f_1^{\theta,\ve},\ldots,f_m^{\theta,\ve})|\le
\sum_{l=1}^m| T_{\sigma}(f_1,\ldots,f_{l-1},f_l-f_l^{\theta,\ve},f_{\theta}^{l+1,\ve},\ldots,f_m^{\theta,\ve}) | .
\]
In view of the    hypothesis that
 $T_\sigma$ is bounded from $L^{p_k^1}\times\cdots\times L^{p_k^m}$ to $L^{q_k}$ and of   \eqref{normftheta}, we obtain
\begin{equation}\label{tendstozero}
\| T_{\sigma}(f_1,\ldots,f_m)- T_{\sigma}(f_1^{\theta,\ve},\ldots,f_m^{\theta,\ve})\|_{L^{q_k}}
\le C(q_0,q_1)\ve
\sum_{l=1}^m\prod_{j\ne l}\big(\|f_j\|_{L^{p^j_k}}^{p^j }+\ve'\big)^{\f 1{{p^j }}}
\end{equation}
and
\[
\| T_{\sigma}(f_1^{\theta,\ve},\ldots,f_m^{\theta,\ve}) \|_{L^{q_k}}\le
C(q_0,q_1)\prod_{l=1}^{m}\big(\|f_l\|_{L^{p^l_k}}^{p^l}+\ve'\big)^{\f{1}{p^l}}
\]
for $k=0,1$.
Using the fact that $\|f\|_{L^q}\le  C\|f\|_{L^{q_0}}^{1-\theta}\|f\|_{L^{q_1}}^{\theta}$,  we
deduce the  estimates
\begin{align*}
  &\| T_{\sigma}(f_1,\ldots,f_m)- T_{\sigma}(f_1^{\theta,\ve},\ldots,f_m^{\theta,\ve})\|_{L^q} \\
  &\le C\ve\Big(\sum_{l=1}^m\prod_{j\ne l}\big(\|f_j\|_{L^{p^j_0}}^{p^j }+\ve'\big)^{\f 1{{p^j }}}\Big)^{1-\theta}
  \Big(\sum_{l=1}^m\prod_{j\ne l}\big(\|f_j\|_{L^{p^j_1}}^{p^j }+\ve'\big)^{\f 1{{p^j }}}\Big)^{\theta},
  \end{align*}
$$
  \| T_{\sigma}(f_1^{\theta,\ve},\ldots,f_m^{\theta,\ve}) \|_{L^q}\le
  C\Big(\prod_{l=1}^{m}\big(\|f_l\|_{L^{p^l_0}}^{p^l}+\ve'\big)^{\f{1}{p^l}}\Big)^{1-\theta}
  \Big(\prod_{l=1}^{m}\big(\|f_l\|_{L^{p^l_1}}^{p^l}+\ve'\big)^{\f{1}{p^l}}\Big)^{\theta}.
$$

These inequalities together with  \eqref{normftheta}, \eqref{Ftheta} and Holder's inequality yield
  \begin{align*}
 & \bigg|{\int T_{\sigma}(f_{1},\ldots,f_{m})g\; dx}\bigg|  \\
  \le &
  C\ve\Big(\sum_{l=1}^m\prod_{j\ne l}\big(\|f_j\|_{L^{p^j_0}}^{p^j }+\ve'\big)^{\f 1{{p^j }}}\Big)^{1-\theta}
  \Big(\sum_{l=1}^m\prod_{j\ne l}\big(\|f_j\|_{L^{p^j_1}}^{p^j }+\ve'\big)^{\f 1{{p^j }}}\Big)^{\theta}\|g\|_{L^{q'}}\\
  &+ C\ve\Big(\prod_{l=1}^{m}\big(\|f_l\|_{L^{p^l_0}}^{p^l}+\ve'\big)^{\f{1}{p^l}}\Big)^{1-\theta}
  \Big(\prod_{l=1}^{m}\big(\|f_l\|_{L^{p^l_1}}^{p^l}+\ve'\big)^{\f{1}{p^l}}\Big)^{\theta}\\
  &+ C K_0^{1-\theta} K_1^{\theta}\sup_{j\in\mathbb Z}
  \norm{(I-\Delta)^{\frac{s}2}[\sigma(2^j\cdot)\widehat{\psi}]}_{L^r}
 \big( \norm{g}_{L^{q'}}^{q'}+\ve'\big)^{\f{1}{q'}}
   \prod_{l=1}^m   \big( \norm{f_l}_{L^{p^l}}^{p^l}+\ve'\big)^{\f{1}{p^l}}.
  \end{align*}
  Now let $\ve$ go to zero, then obtain
  \[
  \bigg|{\int T_{\sigma}(f_{1},\ldots,f_{m})g\; dx}\bigg|\le C K_0^{1-\theta} K_1^{\theta}\sup_{j\in\mathbb Z}
  \norm{(I-\Delta)^{\frac{s}2}[\sigma(2^j\cdot)\widehat{\psi}]}_{L^r}
  \norm{g}_{L^{q'}}
   \prod_{l=1}^m    \norm{f_l}_{L^{p^l}} .
  \]
Taking supremum over all functions $g\in L^{q'}$ yields \eqref{equ:TSgmMul}.

\medskip
\textbf{Case II:} $\bm{\min(q_0,q_1) \le 1}$.

Here we will make use of two following lemmas  proved by Stein and Weiss \cite{SW57}.
\begin{lem}[\cite{SW57}]
\label{USubHar}
  Let $U:\overline{S}\To \mathbb{R}$ be an upper semi-continuous function of admissible growth and subharmonic in the strip $S$. Then for $z_0=x_0+iy_0\in S$ we have
  $$
  U(z_0)\le \int_{-\infty}^{\infty}U\big(i(y_0+y)\big)\omega(1-x_0,y)dy
  + \int_{-\infty}^{\infty}U\big(1+i(y_0+y)\big)\omega(x_0,y)dy,
  $$
  where
  $$
  \omega(x,y) = \dfrac12\dfrac{\sin \pi x}{\cos \pi x+\cosh \pi y}.
  $$
\end{lem}
\begin{lem}[\cite{SW57}]
\label{VSubHar}
Let   $0<c\le 1$.  If the function $V(x,z)$   defined on $\mathbb R^n\times S$   is analytic in $z\in S$ for all  $x \in \mathbb R^n$, and  for all $z\in S$, $V(\cdot , z)$ is integrable and
supported in a set of finite measure, then   $G(z) = \int_{\mathbb{R}^n}\abs{V(x,z)}^cdx$ is continuous and $\log G(z)$ is subharmonic.
\end{lem}

We now continue the proof of the second case. Let fix functions $f_l$ as in the previous case. Choose an integer $\rho>1$ such that $\rho\ge \rho\min(q_0,q_1)>q.$ Take an arbitrary positive simple function $g$ with $\norm{g}_{L^{\rho'}}=1.$ Assume that
$g = \sum_{k=1}^N c_k\chi_{E_k},$ where $c_k>0$ and $E_k$ are pairwise disjoint measurable sets of finite measure. For $z\in\mathbb{C},$ set
\[
  g^z = \sum_{k=1}^Nc_k^{\lambda(z)}\chi_{E_k},\;\;
\textup{where}\;\;
  \lambda(z) = \rho'\left[1-\dfrac{q}{\rho}\left(\dfrac{1-z}{q_0}+\dfrac{z}{q_1} \right)\right].
\]
  Now consider
  $$
  G(z) = \int_{\mathbb R^n} \abs{T_{\sigma_z}(f_{1}^{z,\ve} ,\ldots,f_{m}^{z,\ve})(x)}^{\frac{q}{\rho}}\abs{g^z(x)}\; dx = \sum_{k=1}^N\int_{E_k}
  \abs{c_k^{\lambda(z)}T_{\sigma_z}(f_{1}^{z,\ve},\ldots,f_{m}^{z,\ve})(x)}^{\frac{q}{\rho}}\; dx .
  $$
It follows from Lemma \ref{VSubHar} that $G$ is continuous and $\log G$ is subharmonic. Using H\"older's inequality
  with indices $\frac{\rho q_0}{q} $ and $\big( \frac{\rho q_0}{q} \big)'$ and the fact that the $L^{\rho'}$-norm of $g$ is equal to $1,$ we have
  \[
  \begin{aligned}
  {G(it)}\le& \left\{\int_{\mathbb{R}^n}
  \abs{T_{\sigma_{it}}(f_{1}^{it,\ve} ,\ldots,f_{m}^{it,\ve} )(x)}^{q_0} dx\right\}^{\frac{q}{\rho q_0}}
  \big\| g^{it } \big\|_{L^{ ( \frac{\rho q_0}{q}  )'}}   \\
  \le &\, C \, \Big( (1+\abs{t})^{\frac{mn }{2}+1} \Big)^{\f q\rho}
  \left( K_0 \sup_{j\in Z}\norm{\sigma(2^j\cdot)\widehat{\psi}\,}_{ L^{r}_{s}  }^{\frac{r}{r_0}}
  \prod_{l=1}^m\big(\norm{f_{l}}_{L^{p^{l}}}^{p^l} +\ve'\big)^{\f{1}{p^l}}
  \right)^{\frac{q}{\rho}}.
  \end{aligned}
  \]
  Similarly, we can estimate
  \[
  \begin{aligned}
  {G(1+it)}\le& \left\{\int_{\mathbb{R}^n}
  \abs{T_{\sigma_{it}}(f_{1}^{1+it,\ve},\ldots,f_{m}^{1+it,\ve})(x)}^{q_1} dx\right\}^{\frac{q}{\rho q_1}}
    \big\| g^{1+it } \big\|_{L^{ ( \frac{\rho q_1}{q}  )'}} \\
  \le &\,C \,  \Big( (1+\abs{t})^{\frac{mn }{2}+1} \Big)^{\f q\rho}
  \left( K_1  \sup_{j\in Z}\norm{\sigma(2^j\cdot)\widehat{\psi}\,}_{L^{r}_{s}}^{\frac{r}{r_1}}
  \prod_{l=1}^m\big(\norm{f_{l}}_{L^{p^{l}}}^{p^l}+\ve'\big)^{\f{1}{p^l}}
  \right)^{\frac{q}{\rho}}.
  \end{aligned}
  \]
  By Lemma \ref{USubHar}, we obtain
  \begin{equation}\label{equ:TSgmLe1}
  G(\theta) \le C_*^\prime\,
  \left( K_0^{1-\theta} K_1^\theta \,
  \sup_{j\in Z}\norm{\sigma(2^j\cdot)\widehat{\psi}\,}_{L^{r}_{s}}
  \prod_{l=1}^m\big(\norm{f_{l}}_{L^{p^{l}}}^{p^l}+\ve'\big)^{\f{1}{p^l}}
  \right)^{\frac{q}{\rho}}.
  \end{equation}
  Notice that
  \[
  G(\theta) = \int_{\mathbb{R}^n}
  \abs{T_{\sigma}(f_{1}^{\theta,\ve},\ldots,f_{m}^{\theta,\ve})(x)}^{\frac{q}{\rho}}g(x)\; dx,
  \]
  inequality \eqref{equ:TSgmLe1} implies that
  \begin{align}
  \norm{T_{\sigma}(f_{1}^{\theta,\ve},\ldots,f_{m}^{\theta,\ve})}_{L^q} = &\,
  \norm{\abs{T_{\sigma}(f_{1}^{\theta,\ve},\ldots,f_{m}^{\theta,\ve})}^{\frac{q}{\rho}}}_{L^\rho}^{\frac{\rho}{q}} \notag \\
  =&\,
  \sup
  \set{
  \int \abs{T_{\sigma}(f_{1}^{\theta,\ve},\ldots,f_{m}^{\theta,\ve})(x)}^{\frac{q}{\rho}} dx :\
  g\ge 0, g\, \,\mbox{simple}, \norm{g}_{L^{\rho'}}\le 1
  }^{\frac{\rho}{q}}  \notag    \\
  \le&\, (C_*^\prime )^{\frac{\rho}{q}}\,  K_0^{1-\theta} K_1^\theta \sup_{j\in Z}\norm{\sigma(2^j\cdot)\widehat{\psi}}_{L^{r}_{s}}
  \prod_{l=1}^m\big(\norm{f_{l}}_{L^{p^{l}}}^{p^l}+\ve'\big)^{\f{1}{p^l}}. \label{aiwueha}
  \end{align}
  Finally, we write
  $$
  \| T_\si(f_1,\dots , f_m) \|_{L^q} \le C(q)    \| T_\si(f_1,\dots , f_m)-T_\si(f_1^{\theta,\ve} ,\dots , f_m^{\theta,\ve}) \|_{L^q} + C(q)
  \|T_\si(f_1^{\theta,\ve} ,\dots , f_m^{\theta,\ve}) \|_{L^q}
  $$
  and we note that for the second term we use \eqref{aiwueha}, while   the first term tends
  to zero by   \eqref{tendstozero}.
Letting $\ve\to 0$,  we deduce \eqref{equ:TSgmMul}.
\end{proof}

Note that the proof of Theorem \ref{MulInterp} is much simpler in the case $r_0=r_1=2$, and this was proved earlier  in \cite[Theorem 6.1, Step 1]{GMT13}; see also \cite[Theorem 2.3]{GN16}. In this case, the domains can be arbitrary  Hardy spaces. Here is the statement in this case:
\begin{thm}[\cite{GMT13}]\label{thm:MulInterR2}
  Let $p_0^l,p_1^l,p^l,q_0,q_1,q, s_0,s_1,s$ and $\theta$ be positive as in Theorem \ref{MulInterp} for $l=1,\ldots,m$. Assume that $s_0,s_1>\frac{mn}{2}$ and that
    \begin{equation*}
    \norm{T_\sigma(f_1,\ldots,f_m)}_{L^{q_k}(\R^n) }\le K_k\sup_{j\in\mathbb Z}\norm{\sigma(2^j\cdot)\widehat{\Psi}}_{L^{2}_{s_k}(\R^{mn})}
    \prod_{l=1}^m
    \norm{f_l}_{H^{p^{l}_k}(\R^n)}
  \end{equation*}
  for $k=0,1$ where $K_0,K_1$ are positive constants.
  Then we have the intermediate estimate:
  \begin{equation*}
    \norm{T_\sigma(f_1,\ldots,f_m)}_{L^{q}(\R^{ n})}\le C_* \, K_0^{1-\theta}K_1^{\theta}
    \sup_{j\in\mathbb Z}\norm{\sigma(2^j\cdot)\widehat{\Psi}}_{L^{2}_{s}(\R^{mn})}
    \prod_{l=1}^m\norm{f_l}_{H^{p^{l}}(\R^{ n})}
  \end{equation*}
  where $C_*$ depends on all the indices, $\theta$, and the dimension.

\end{thm}

\section{The proof of the main result via interpolation}
We now turn to the proof of Theorem~\ref{main}.

\begin{proof}
(a)  Assume $n/2<s\le n$ and let
\[
\Gamma_1 = \Big\{
\Big(\frac{1}{p_1},\frac{1}{p_2}\Big)\ :\  \frac1{p_1}<\frac{s}{n}, \frac1{p_2}<\frac{s}{n}, 
1-\frac{s}{n}<\frac1{p}=\f{1}{p_1}+\f{1}{p_2}<\frac{s}{n}+\frac12
\Big\} . 
\]
We will prove that
  \begin{equation}\label{EQ!TSE}
  \|T_{\sigma}(f_1,f_2)\|_{L^p(\mathbb{R}^n)}\le C\sup_{j\in\mathbb{Z}}
  \|\sigma(2^j\cdot)\widehat{\Psi}\|_{L^r_{s}(\mathbb{R}^{2n})}
  \|f_1\|_{L^{p_1}(\mathbb{R}^n)}\|f_2\|_{L^{p_2}(\mathbb{R}^n)}
\end{equation}
for every $(\frac{1}{p_1},\frac{1}{p_2})\in \Gamma_1$,
which is a convex set with vertices $D,K,L,G,H$ and $N$ (see Figure \ref{FigA} below).
By   multilinear real interpolation \cite[Corollary 7.2.4]{Grafa14MFA}, we only need to verify the boundedness of $T_\sigma$ at   points in $\Gamma_1$ near its vertices  $D,K, L, G, H,N $ which do not lie in $ \Gamma_1$.

\begin{figure}[ht]
    \centering
    \begin{subfigure}[b]{0.4\textwidth}
        \includegraphics[width=\textwidth]{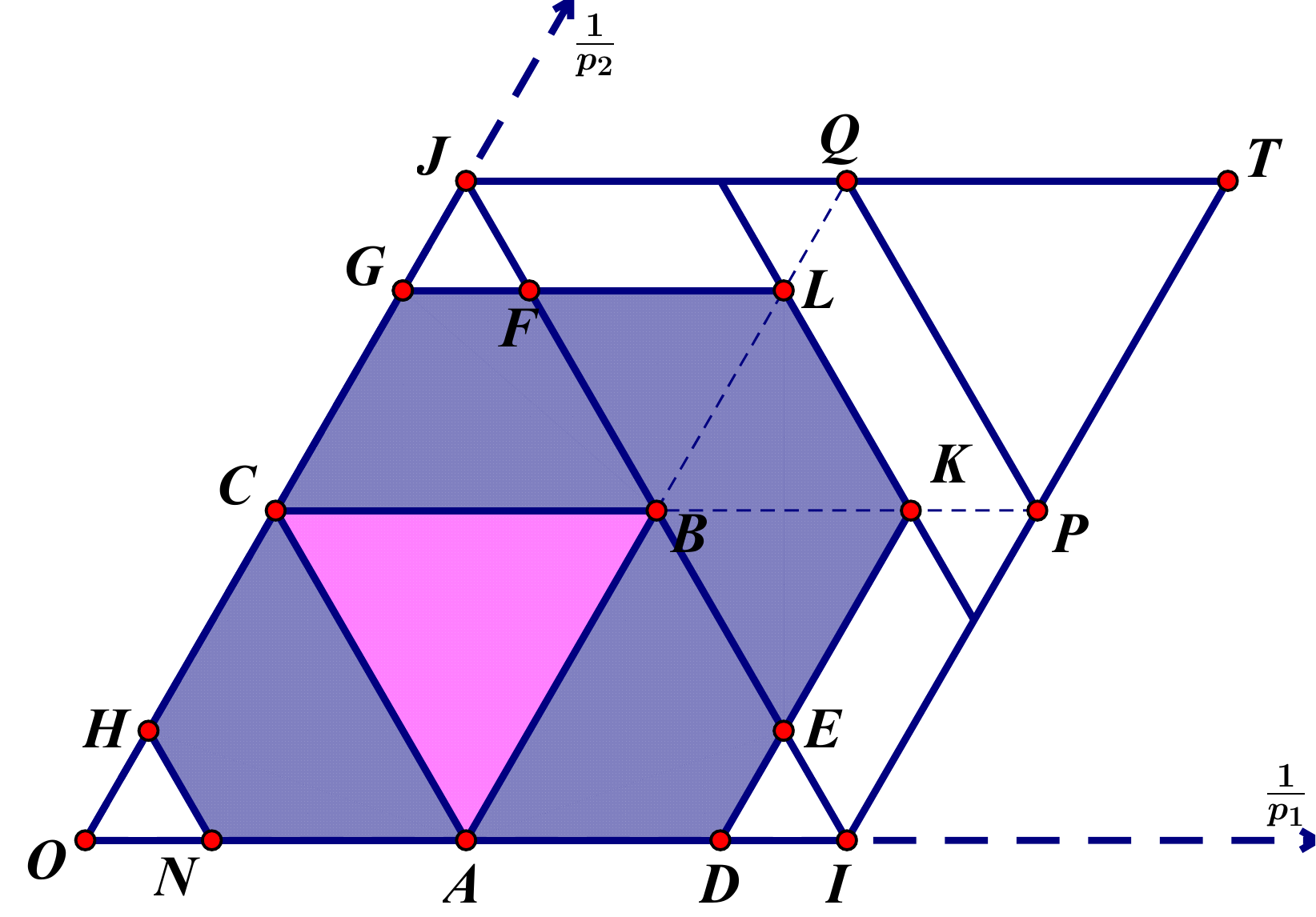}
        \caption{$\frac{n}{2}<s\le n$}
        \label{FigA}
    \end{subfigure}
    \quad 
    \begin{subfigure}[b]{0.4\textwidth}
        \includegraphics[width=\textwidth]{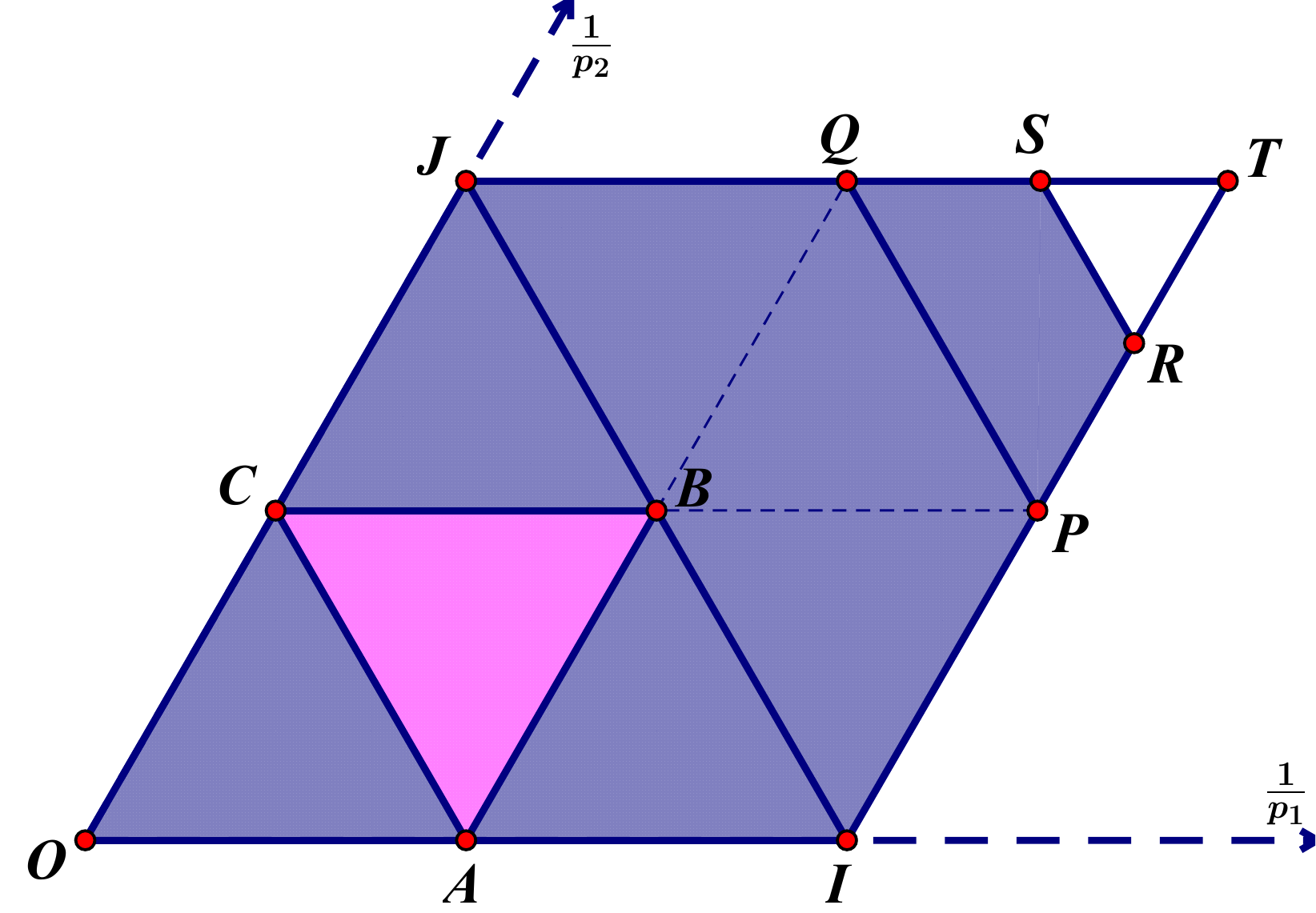}
        \caption{$n<s\le \frac{3n}{2}$}
        \label{FigB}
    \end{subfigure}
    \caption{Boundedness holds in the shaded regions and unboundedness in the white regions.
    The local $L^2$ region is shaded in a lighter color.}
    \label{Fig01}
\end{figure}

When $1\le p<\infty$, by duality, if $T_{\sigma}$ maps $L^{p_1}\times L^{p_2}\to L^p$, then it also maps $L^{p'}\times  L^{p_2}\to L^{p_1'}$, since the H\"ormander condition  $\sup_{j\in\mathbb{Z}}\|\sigma(2^j\cdot)\widehat{\Psi}\|_{L^r_s(\mathbb{R}^{2n})}$ is invariant under   duality; see   \cite{Grafa14MFA,GS12}.
Therefore, if $T_{\sigma}$ is bounded near $D$, then $T_{\sigma}$ is also bounded near $N$ by  duality. By symmetry, if $T_\sigma$ is bounded near $N,D$ and $K$ then it is bounded near $H, G$ and $L$ as well. From these reductions, it remains to prove \eqref{EQ!TSE} at   points in $\Gamma_1$ near $D$ and $K$. 

With $s_1>\frac{n}2$ and $r_1s_1>2n$, we recall the following   \cite[Theorem 1]{GHH2016II}:
\begin{equation}
\label{eq-019B}
  \|T_{\sigma}(f_1,f_2)\|_{L^1(\mathbb{R}^n)}\le C\sup_{j\in\mathbb{Z}}
  \|\sigma(2^j\cdot)\widehat{\Psi}\|_{L^{r_1}_{s_1}(\mathbb{R}^{2n})}
  \|f_1\|_{L^{2}(\mathbb{R}^n)}\|f_2\|_{L^{2}(\mathbb{R}^n)}.
\end{equation}
By duality it follows from \eqref{eq-019B}   that when $s_1>\frac{n}2$ and $r_1s_1>2n$
we have 
\begin{equation}
\label{eq-019A}
  \|T_{\sigma}(f_1,f_2)\|_{L^2(\mathbb{R}^n)}\le C\sup_{j\in\mathbb{Z}}
  \|\sigma(2^j\cdot)\widehat{\Psi}\|_{L^{r_1}_{s_1}(\mathbb{R}^{2n})}
  \|f_1\|_{L^{2}(\mathbb{R}^n)}\|f_2\|_{L^{\infty}(\mathbb{R}^n)}.
\end{equation}
Theorem 1.1 in  \cite{MT13} (with $  s_1=s_2$ in \cite{MT13} being $\gamma$ below) implies  that
\[
\|T_{\sigma}(f_1,f_2)\|_{L^q(\mathbb{R}^n)} \le
  C\sup_{j\in\mathbb{Z}}\|(I-\Delta_{\xi_1})^{\frac{\gamma}2}(I-\Delta_{\xi_2})^{\frac{\gamma}2}
  \big[\sigma(2^j\cdot)\widehat{\Psi}\big]\|_{L^2(\mathbb{R}^{2n})}
  \|f_1\|_{L^{q_1}(\mathbb{R}^n)} \|f_2\|_{L^{q_2}(\mathbb{R}^n)}
\]
for $\gamma>\frac{n}{2}$, where $1< q_1,q_2\le \infty$, $\frac{1}{q}=\frac{1}{q_1}+\frac{1}{q_2}< \frac{2\gamma}{n}+ \frac{1}{2}$. Given $s_2>n$, choose $\gamma=\frac{s_2}{2}>\frac{n}{2}$ and observing the trivial estimate 
\begin{equation*}\label{EQ!SobN}
\sup_{j\in\mathbb{Z}}\|(I-\Delta_{\xi_1})^{\frac{\gamma}2}(I-\Delta_{\xi_2})^{\frac{\gamma}2}
\big[\sigma(2^j\cdot)\widehat{\Psi}\big]\|_{L^2(\mathbb{R}^{2n})}
\le C \sup_{j\in\mathbb{Z}}\|\sigma(2^j\cdot)\widehat{\Psi}\|_{L^2_{s_2}(\mathbb{R}^{2n})},
\end{equation*}
we obtain
\begin{equation}\label{PtI}
\|T_{\sigma}(f_1,f_2)\|_{L^q(\mathbb{R}^{ n})} \le
  C\sup_{j\in\mathbb{Z}}\|\sigma(2^j\cdot)\widehat{\Psi}\|_{L^2_{s_2}(\mathbb{R}^{2n})}
  \|f_1\|_{L^{q_1}(\mathbb{R}^{ n})} \|f_2\|_{L^{q_2}(\mathbb{R}^{ n})}
\end{equation}
for all  $1< q_1,q_2\le \infty$, $\frac{1}{q}=\frac{1}{q_1}+\frac{1}{q_2}< \frac{s_2}{n}+ \frac{1}{2}$.

We now use Theorem \ref{MulInterp} to interpolate 
between \eqref{eq-019A} and \eqref{PtI} (for $q_1=q$ near $1$ and  $q_2=\infty$).    
We obtain \eqref{EQ!TSE} at   points $D_1(\frac{1}{p_1},0)$  with $\frac1{p_1}<\frac{s}{n}$ which are near the point $D(\frac{s}{n},0)$. Similarly, interpolating between \eqref{eq-019B} and \eqref{PtI} ($q_1$ near 1, $q_2=2$) yields \eqref{EQ!TSE} at   points  $K_1(\frac1{p_1},\frac12)$  with $\frac1{p_1}<\frac{s}n$ 
 near   $K(\frac{s}n,\frac12)$. This 
yields \eqref{EQ!TSE} on $\Gamma_1$ and 
 completes  part (a).

(b) Assume $n<s\le \frac{3n}2$. Since $r\ge 2$, the Kato-Poince inequality \cite{GO} implies that
\begin{equation}
\label{eq-KP}
\sup_{j\in\mathbb{Z}}
  \|\sigma(2^j\cdot)\widehat{\Psi}\|_{L^2_{s}(\mathbb{R}^{2n})}
  \lesssim
  \sup_{j\in\mathbb{Z}}
  \|\sigma(2^j\cdot)\widehat{\Psi}\|_{L^r_{s}(\mathbb{R}^{2n})}.
\end{equation}
Combining   estimates \eqref{eq-KP} and \eqref{PtI} yields \eqref{EQ!TSE} 
in the open  pentagon $OIRSJ$ union the open segments $OI$ and $OJ$. 
 This completes the second part of Theorem \ref{main}.

(c) In the last case when $s>\frac{3n}{2}$, notice that 
 condition \eqref{0099LE} reduces to $p> \frac12$
and since 
\[
\sup_{j\in\mathbb{Z}}
  \|\sigma(2^j\cdot)\widehat{\Psi}\|_{L^r_{\frac{3n}{2}}(\mathbb{R}^{2n})}
  \le
\sup_{j\in\mathbb{Z}}
  \|\sigma(2^j\cdot)\widehat{\Psi}\|_{L^r_s(\mathbb{R}^{2n})} , 
\]
the case in part (b) applies and yields \eqref{EQ!TSE} for every point in 
the entire   rhombus $OITJ$ union the open segments $OI$ and $OJ$. 
The proof of Theorem~\ref{main} is now complete.
\end{proof}

\section{An application}

We consider the following multiplier on $\mathbb R^{2n}$:
$m_{a,b}(\xi_1,\xi_2) = \psi(\xi_1,\xi_2) |(\xi_1,\xi_2)|^{- b} e^{i|(\xi_1,\xi_2)|^a}$
where $a > 0$, $a \neq   1$, $b > 0$, and $\psi$ is a smooth function  on $\mathbb R^{2n}$ which vanishes in a
neighborhood of the origin and is equal to $1$ in a neighborhood of infinity. One can verify that $m_{a,b}$
satisfies \eqref{1H} on $\mathbb R^{2n}$ with $s = b/a$ and any $r>2n/s$. 

The range of $p$'s for which  $m_{a,b}$ is a bounded bilinear 
multiplier on $L^p(\mathbb R^{2n})$
can be completely described by the equation  $   |\f 1p-\f 12|\le \f{b/a}{2n} $ (see  Hirschman
 \cite[comments   after Theorem 3c]{hirschman2},    Wainger~\cite[Part II]{W}, and    Miyachi~\cite[Theorem 3]{Miy});   similar  examples 
 of multipliers of limited boundedness 
 are contained in Miyachi and Tomita
\cite[Section 7]{MT13}.

As a consequence of Theorem~\ref{main} we obtain that the bilinear multiplier operator
associated with $m_{a,b}$ is bounded from $L^{p_1}(\mathbb R^n)\times L^{p_2}(\mathbb R^n)$ to $L^{p }(\mathbb R^n)$ in the following cases:
\begin{enumerate}[(i)]
\item when $n\ge b/a> n/2$ and
\[
\frac{1}{p_1}<\frac{b}{an},\,
\frac{1}{p_2}<\frac{b}{an},\,
1-\frac{b}{an}<\frac{1}{p}<\frac{b}{an}+\frac12.
\]
\item when $3n/2\ge b/a> n $ and
$$
 \f{1}{p }  <\f {b}{an}+\f12\, ;
$$
\item when   $b/a >3n/2$ in the entire range   of exponents $1<p_1,p_2\le \infty$, 
$\frac12<p<\infty$.
\end{enumerate}

The boundedness of this specific
bilinear multiplier is   unknown to us outside the above range of indices.


\begin{thebibliography}{99}

\bibitem{CT77} A.~P. Calder{\'o}n and A.~Torchinsky, \emph{Parabolic maximal functions
   associated with a distribution, {II}}, Adv. in Math. \textbf{24} (1977),  101--171.

\bibitem{FT12}
M.~Fujita and N.~Tomita, \emph{Weighted norm inequalities for multilinear
  {Fourier} multipliers}, Trans. Amer. Math. Soc. \textbf{364} (2012),
  6335--6353.

\bibitem{Grafa14CFA}
L.~Grafakos, \emph{Classical Fourier Analysis}, 3rd ed., Graduate Texts in
  Mathematics, GTM 249, Springer-Verlag, New York, 2014.

\bibitem{Grafa14MFA}
L.~Grafakos, \emph{Modern Fourier Analysis}, 3rd ed., Graduate Texts in
  Mathematics, GTM 250, Springer-Verlag, New York, 2014.

 \bibitem{GHH2016I}
L.~Grafakos, D.~He, and P.~Honz{\'\i}k, H. V. Nguyen, \emph{The {H\"ormander} multiplier
  theorem {I}: The linear case},
  Illinois J. Math. \textbf{61} (2017),   25--35. 

\bibitem{GHH2016II}
L.~Grafakos, D.~He, and P.~Honz{\'\i}k, \emph{The {H\"ormander} multiplier
  theorem {II}: The local $L^2$  case}, 
   Math. Zeit. \textbf{289} (2018),  875--887. 

\bibitem{GMNT16}
L.~Grafakos, A.~Miyachi, H.~V. Nguyen, and N.~Tomita, \emph{Multilinear
  {Fourier} multipliers with minimal {Sobolev} regularity, {II}}, J. Math. Soc.
  Japan \textbf{69} (2017),   529--562.

\bibitem{GMT13}
L.~Grafakos, A.~Miyachi, and N.~Tomita, \emph{On multilinear {Fourier}
  multipliers of limited smoothness}, Can. J. Math. \textbf{65} (2013),   299--330.

\bibitem{GN16}
L.~Grafakos and H.~V. Nguyen, \emph{Multilinear {Fourier} multipliers with
  minimal {Sobolev} regularity, {I}}, Colloquium Math. \textbf{144} (2016),  1--30.

\bibitem{GO} L. Grafakos and S. Oh, \emph{The Kato-Ponce inequality}, Comm. in PDE \textbf{39} (2014),   1128--1157.

\bibitem{GS12}
L.~Grafakos and Z.  Si, \emph{The {H\"ormander} multiplier theorem for
  multilinear operators}, J. Reine Angew. Math. \textbf{668} (2012), 133--147.

 
\bibitem{hirschman2} I. I. Jr. Hirschman,
\emph{On multiplier transformations},  Duke Math. J. {\bf 26} (1959), 221--242.

\bibitem{Hoe}  L. H\"ormander,
\emph{Estimates for translation invariant operators
in $L^p$ spaces,} Acta Math. {\bf 104} (1960), 93--139.

\bibitem{KP} T. Kato, G. Ponce, \emph{Commutator estimates and the Euler and Navier-Stokes equations}, Comm.
Pure App. Math. \textbf{41} (1988),  891--907.

\bibitem{Mihl}
S.~G. Mikhlin, \emph{On the multipliers of {Fourier} integrals}, Dokl. Akad.
  Nauk SSSR (N.S.) \textbf{109} (1956),   701--703.

\bibitem{Miy} A. Miyachi,   \emph{On some Fourier multipliers for $H^p(\mathbb R^n)$},  J. Fac. Sci. Univ. Tokyo
Sect. IA Math. {\bf 27} (1980),    157--179. %

\bibitem{MT13}
A.~Miyachi and N.~Tomita, \emph{Minimal smoothness conditions for bilinear
  {Fourier} multipliers}, Rev. Mat. Iberoamer. \textbf{29} (2013),   495--530.

\bibitem{MT14}
A.~Miyachi and N.~Tomita, \emph{Boundedness criterion for bilinear {Fourier} multiplier
  operators}, Tohoku Math. J. \textbf{66} (2014),   55--76.


\bibitem{slav} L. Slav\'\i kov\'a,
personal communication. 


\bibitem{SW57}
E.~M. Stein and G.~Weiss, \emph{On the interpolation of analytic families of
  operators acting on {$H^{p}$}-spaces}, Tohoku Math. J. \textbf{9} (1957),  318--339.

\bibitem{Tomit10}
N.~Tomita, \emph{A {H\"ormander} type multiplier theorem for multilinear
  operators}, J. Funct. Anal. \textbf{259} (2010),   2028--2044.


  \bibitem{W} S. Wainger,  \emph{Special trigonometric series in k-dimensions},  Mem. Amer. Math. Soc.
{\bf 59} (1965), 1--102.


\end{thebibliography}

\end{document}